\documentclass[12pt]{article}

\usepackage{amsmath,amsthm,amsfonts,latexsym,amscd,amssymb}

\def\qed{{\hbadness=10000\hfill\ \vbox{\hrule height.09ex
     \hbox{\vrule width.09ex height1.55ex depth.2ex \kern1.8ex
     \vrule width.09ex height1.55ex depth.2ex}\hrule height.09ex}\break
     \bigskip}}
\newtheorem{theorem}{Theorem}[section]
\newtheorem{lemma}[theorem]{Lemma}
\newtheorem{corollary}[theorem]{Corollary}
\newtheorem{proposition}[theorem]{Proposition}

\newtheorem{remark}[theorem]{Remark}

\begin{document}

\title{On the Number of Isomorphism Classes of Transversals}

%----------Author 1
\author{Vipul Kakkar\footnote{The first author is supported by CSIR, Government of India.}
~ and R.P. Shukla\\
Department of Mathematics, University of Allahabad \\
Allahabad (India) 211 002\\
Email: vplkakkar@gmail.com; shuklarp@gmail.com}
\date{}

\maketitle


\begin{abstract}
In this paper we prove that there does not exist a subgroup $H$ of a finite group $G$ such that the number of isomorphism classes of normalized right transversals of $H$ in $G$ is four.
\end{abstract}
\noindent \textbf{\textit{Mathematical Subject Classification (2010):}}{20D60, 20N05} \\
\noindent \textbf{\textit{Key words:}} Group torsion, Right loop, Normalized right transversal
%%% ----------------------------------------------------------------------
%%% ----------------------------------------------------------------------
%\tableofcontents
\section{Introduction and Statement of the Main Result}
Let $G$ be a finite group and $H$ a subgroup of $G$. Let $S$ be a normalized right transversal (NRT) of $H$ in $G$, that is $S$ is a subset of $G$ obtained by choosing one and only one element from each right coset of $H$ in $G$ and $1 \in S$. Then $S$ has an induced binary operation $\circ$ given by $\{x \circ y\}=Hxy \cap S$, with respect to which $S$ is a right loop with identity $1$, that is, a right quasigroup with both sided identity  (see \cite[Proposition 2.2, p.42]{smth},\cite{rltr}). Conversely, every right loop can be embedded as an NRT in a group with some universal property (see \cite[Theorem 3.4, p.76]{rltr}). Let $\langle S \rangle$ be the subgroup of $G$ generated by $S$ and $H_S$ be the subgroup $\langle S \rangle \cap H$. Then  $H_S=\langle \left\{xy(x \circ y)^{-1}|x, y \in S \right\} \rangle$ and $H_{S}S=\langle S \rangle$ (see \cite{rltr}).

Identifying $S$ with the set $H \backslash G$ of all right cosets of $H$ in $G$, we get a transitive permutation representation $\chi_{S}:G\rightarrow Sym(S)$ defined by $\left\{\chi_{S}(g)(x)\right\}=Hxg \cap S, g\in G, x\in S$. The kernal $Ker \chi_S$ of this action is $Core_{G}(H)$, the core of $H$ in $G$.   

Let $G_{S}=\chi_{S}(H_{S})$. This group is known as the \textit{group torsion} of the right loop $S$ (see \cite[Definition 3.1, p.75]{rltr}). The group $G_S$ depends only on the right loop structure $\circ$ on $S$ and not on the subgroup $H$. Since $\chi_S$ is injective on $S$ and if we identify $S$ with $\chi_S(S)$, then $\chi_S(\langle S \rangle)=G_SS$ which also depends only on the right loop $S$ and $S$ is an NRT of $G_S$ in $G_SS$. One can also verify that $ Ker(\chi_S|_{H_SS}: H_SS \rightarrow G_SS)=Ker(\chi_S|_{H_S}: H_S \rightarrow G_S)=Core_{H_SS}(H_S)$ and $\chi_S|_S$=the identity map on $S$. If $H$ is a corefree subgroup of $G$, then there exists an NRT $T$ of $H$ in $G$ which generates $G$ (see \cite{cam}). In this case, $G=H_TT \cong G_TT$ and $H=H_T \cong G_T$. Also $(S, \circ)$ is a group if and only if $G_S$ trivial.

Let $\mathcal{T}(G, H)$ denote the set of all normalized right transversals (NRTs) of $H$ in $G$. We say that $S$ and $T \in \mathcal{T}(G, H)$ are isomorphic (denoted by $S \cong T$), if their induced right loop structures are isomorphic. Let $\mathcal{I}(G,H)$ denote the set of isomorphism classes of NRTs of $H$ in $G$.  
  
In \cite[Main Theorem, p.643]{rjpf}, it is shown that $|\mathcal{I}((G,H))|=1$ if and only if $H\trianglelefteq G$. It is obtained in \cite[Theorem, p. 1718]{viv1} that there is no pair $(G,H)$ such that $|\mathcal{I}(G,H)|=2$. It is easy to observe that if $H$ is a non-normal subgroup of $G$ of index 3, then $|\mathcal{I}(G,H)|=3$. The converse of this statement is proved in \cite[Theorem A, p. 2025]{viv2}. Also, it is shown in \cite[Theorem 3.7, p.2693]{rpsc} that if $T_n$ denotes the number of non-isomorphic right loops of order $n$, then $\left|\mathcal{I}(Sym(n),Sym(n-1)) \right|=T_n$, where $Sym(m)$ denotes the symmetric group on $m$ symbols. Moreover, if there is any pair $(G,H)$ such that the index $[G:H]$ of $H$ in $G$ is $n$ and if $|\mathcal{I}(G,H)|=T_n$, then there is a surjective homomorphism $\psi:G\rightarrow Sym(n)$ such that $\psi(H)=Sym(n-1)$ and $\psi^{-1}(Sym(n-1))=H$ (see \cite[Proposition 3.8, p.2694]{rpsc}).

Let $Aut_{H}G$ denote the group of all automorphisms of $G$ taking $H$ onto $H$.
%The group $Aut_{H}G$ acts on each isomorphism class of NRTs of $H$ in $G$. Discussion before Remark 2.26 of \cite[p.2686]{rpsc} shows that if $S, T \in \mathcal{T}(G,H)$ and if $\theta: S \rightarrow T$ is an isomorphism, then it induces an isomorphism $\tilde{\theta}$ from $G_SS$ to $G_TT$, which takes $G_S$ to $G_T$ and $\tilde{\theta}|_S=\theta$. Thus if $S \in \mathcal{T}(G,H)$ generates $G$ and $H$ is a core-free subgroup of $G$, $\chi_S:G\rightarrow G_SS$ is an isomorphism which takes $H$ into $G_S$ and $\chi_S|_S$=the identity map on $S$. Therefore, if $H$ is corefree subgroup of $G$, then the group $Aut_HG$ acts transitively on the set $\{T \in \mathcal{T}(G,H)|T \cong S ~and~ \langle S \rangle=G \}$.   
The group $Aut_HG$ acts on each isomorphism class in $\mathcal{T}(G,H)$. Thus the number of non-isomorphic right loops is at most the number of orbits of the action of $Aut_{Sym(n-1)}Sym(n)$ on $\mathcal{T}(Sym(n),Sym(n-1))$. Clearly $Core_{Sym(n)}(Sym(n-1))=\left\{1\right\}$. For $n \neq 6, Aut Sym(n)=Inn (Sym(n)) \cong Sym(n)$ (see \cite[Proposition 2.18, p.300]{suz}) and $Aut Sym(6)=Inn(Sym(6)) \rtimes C_2$ (see \cite[Proposition 2.19, p.300]{suz}), where $C_n$ denotes the cyclic group of order $n$. It can be checked that $Aut_{Sym(n-1)}Sym(n) \cong Sym(n-1) \leq Inn(Sym(n))$ for all $n$. It follows from the proof of \cite[Theorem 3.7, p. 2693]{rpsc} that binary operations of $S$ and $T$ define an element $\alpha \in Sym(n)$  such that $\alpha(1)=1$ and $\alpha S \alpha^{-1}=T$. Which means that the number of orbits of the action of $Aut_{Sym(n-1)}Sym(n)$ on $\mathcal{T}(Sym(n),Sym(n-1))$ is precisely the number of isomorphism classes in $\mathcal{T}(Sym(n),Sym(n-1))$. The same is true for the pair $(Alt(n),Alt(n-1))$, where  $Alt(m)$ denotes the alternating group of degree $m$ (since $Aut(Alt(n))\cong Aut(Sym(n))$, $Aut_{Alt(n-1)}Alt(n) \cong Sym(n-1) \leq Inn(Sym(n)))$.

Using GAP (\cite{gap}), we have calculated the number of orbits of the action by conjugation of $Sym(n-1)$ on $\mathcal{T}(Sym(n), Sym(n-1))$ for $n=4$ and $5$. These are 44 and 14022 respectively. In \cite{viv3}, an explicit formula for the number of orbits of the conjugation action of $Sym(n-1)$ on $\mathcal{T}(Sym(n),Sym(n-1))$ has been obtained.  
%\begin{center}
%\begin{tabular}{|l|r|}
%\hline
%$n$ & Number of non-isomorphic right loops\\
%\hline
%1 & 1\\
%2 & 1\\
%3 & 3\\
%4 & 44\\
%5 & 14022\\
%\hline
%\end{tabular}
%\end{center}
If $H$ has non-trivial core, then the number of $Aut_HG$-orbits in $\mathcal{T}(G,H)$ may be different from the number $|\mathcal{I}(G,H)|$. For example, let $G_1= Sym(4)$ and $H_1=\langle \{(1,3),(1,2,3,4)\} \rangle \cong D_8$, where $D_{2n}$ denotes the dihedral group of order $2n$. Then NRTs $\{I,(3,4),(2,3)\}$, $\{I,(3,4),(2,3,4)\}$, $\{I,(3,4),(1,2,3,4)\}$ and $\{I,(2,4,3),(2,3,4)\}$ to $H_1$ in $G_1$, where $I$ is identity permutation, lie in different orbits of $Aut_{H_{1}}G_{1}$ (as the set of orders of group elements in any two NRTs are not same). However, since $H_1$ is a non-normal subgroup of $G_1$ of index 3, $\left|\mathcal{I}(G_1,H_1)\right|=3$.

Let $N=Core_G(H)$. Clearly $L \mapsto \nu (L)=\{Nx~|~x\in L\}$, where $\nu$ is the quotient map from $G$ to $G/N$, is a surjective map from $\mathcal{T}(G,H)$ to $\mathcal{T}(G/N,H/N)$ such that the corresponding NRTs are isomorphic. 

%We can also calculate the number of isomorphism classes for the pair $(G,H)$ of group $G$ and  core free subgroup $H$ when $Aut_HG$ acts on $\mathcal{T}(G,H)$ by conjugation. In this case $\left| \mathcal{I}(G,H) \right|$ is the number of orbits of the action. For example, since $Aut(Alt(n))\cong Aut(Sym(n))$, $Aut_{Alt(n-1)}Alt(n) \cong Sym(n-1) \leq Inn(Sym(n))$, where  $Alt(n)$ denotes the alternating group of degree $n$. %Using GAP, we have calculated the number of orbits of the action of $Aut_HG$ on $\mathcal{T}(G,H)$ for certain pairs $(G,H)$ of group and core free subgroup, which is precisely the number of isomorphism classes $\left| \mathcal{I}(G,H) \right|$. For example, $\left|\mathcal{I}(Alt(4),Alt(3))\right|=7, \left|\mathcal{I}(D_8,C_2)\right|=6, \left|\mathcal{I}(Alt(4),C_2)\right|=5$, and $\left|\mathcal{I}(G_1,H_1)\right|=54$, where $G_1 \cong S_3 \times C_3$ and $H_1 \cong C_3$ (a non-normal subgroup of $G_1$). 

 Let $X$ denote the set of all pairs $(G,H)$, where $G$ is a finite group and $H$ a subgroup of $G$. In view of the above discussion, it seems an interesting problem to find the image set and the inverse image set of the map $\varphi: X \rightarrow \mathbb{N}$ defined by $\varphi((G,H))=\left|\mathcal{I}(G,H)\right|$.

In this paper, we prove that $4\notin Image(\varphi)$, that is, we prove the following theorem:
%\begin{equation}\label{testequation}
%\text{This is a sample equation: } c^2=a^2+b^2
%\end{equation}
%$ \phi :\left\{(G,H)|G~is~a~finite~group~and~H~a~ subgroup\right\}\rightarrow \mathbf{N}$ defined by $\phi((G,H))$\\$=\left|\mathcal{I}(G,H)\right|$.

\begin{theorem}[Main Theorem]\label{mt}
Let $G$ be finite group and $H$ be a subgroup of $G$. Then $|\mathcal{I}(G,H)| \neq 4$.
%\eqref{testequation}.
\end{theorem}
The proof of the Theorem \ref{mt} is based by the method of contradiction and essentially uses the same techniques of \cite{rjpf}. Assuming the falsity of the result, we can find a pair $(G,H)$, to be called as a \textit{minimal counterexample} such that
\smallskip

\noindent \textbf{(i)}  $|G|$ is minimal.\\
%\noindent
\noindent \textbf{(ii)} the index $[G:H]$ is minimal and\\
%\noindent
\noindent \textbf{(iii)}$|I(G, H)| = 4$.

We will study various properties of a minimal counterexample and come to the case of a finite non-abelian simple group. With the knowledge of the order of automorphism groups of finite non-abelian simple groups, we will derive a contradiction. Unfortunately, we do not have an alternate proof of the Main Theorem where the use of the classification of finite simple groups could be avoided. However in \cite{vk}, we now have a short proof of the \cite[Main Theorem, p.643]{rjpf} where the classification of finite simple groups could be avoided. %We also hope that the theorem is true for $2^n$. 

%\section{Preliminaries}
%Let $S$ be an NRT of a subgroup $H$ of finite group $G$ and $\circ$ be the induced binary operation on $S$. 

\section{Properties of a Minimal Counterexample}
\begin{proposition}\label{p1} Let $(G,H)$ be a minimal counterexample. Then 

(i) $Core_G(H)=\left\{1\right\}$,

(ii) if $S \in \mathcal{T}(G,H)$ such that $\langle S \rangle = G$, then there exists an isomorphism $f:G \rightarrow G_SS$ which takes $H$ onto $G_S$ and fixes $S$ elementwise,
 
(iii) if $S \in \mathcal{T}(G,H)$ such that $\langle S \rangle \neq G$, then $H_S=\left\{1\right\}.$  
\end{proposition}

\begin{proof} (i) Let $N=Core_G(H)$. Assume that $N \neq \left\{1\right\}$. Since the quotient map $\nu:G \rightarrow G/N$ induces a surjective correspondence between $\mathcal{T}(G,H)$ and $\mathcal{T}(G/N,H/N)$ such that the corresponding right loops are isomorphic, $|\mathcal{I}(G/N,H/N)|<4$, for $(G,H)$ is a minimal counterexample.

If $|\mathcal{I}(G/N,H/N)|=1$, then $H/N \trianglelefteq G/N$ (see \cite[Main Theorem, p.643]{rjpf}) and so $H \trianglelefteq G$, a contradiction. By \cite[Theorem, p. 1718]{viv1}, $\left|\mathcal{I}(G,H)\right| \neq 2$. If $|\mathcal{I}(G/N,H/N)|=3$, then $[G:H]=[G/N:N/H]=3$ (see \cite[Theorem A, p.2025]{viv2}). But in this case, $|\mathcal{I}(G,H)|=3$, which is a contradiction.

(ii) Let $S \in \mathcal{T}(G,H)$ such that $\langle S \rangle=G$. Then $H_S=H$. Let $\chi_{S}:H_SS\rightarrow G_SS$ be the surjective homomorphism defined as in the second paragraph of the Section $1$. Then $Ker \chi_{S}=\{1\}$ (by (i)). Hence $\chi_{S}$ is an isomorphism which takes $H$ onto $G_S$ and fixes $S$ elementwise.

(iii) Let $S \in \mathcal{T}(G,H)$. Assume that $H_SS=\langle S \rangle \neq G$. Since $\mathcal{T}(\langle S \rangle, H_S) \subseteq \mathcal{T}(G,H)$ and $(G,H)$ is a minimal counterexample, $|\mathcal{I}(\langle S \rangle, H_S)|<4$. By \cite[Theorem, p. 1718]{viv1}, $\left|\mathcal{I}(H_SS,H_S)\right| \neq 2$ and by \cite[Theorem A, p.2025]{viv2}, $\left|\mathcal{I}(H_SS,H_S)\right| \neq 3$. Hence $\left|\mathcal{I}(\langle S \rangle,H_S)\right| = 1$. Thus, by \cite[Main Theorem, p.643]{rjpf} $H_S \trianglelefteq \langle S \rangle$. Let $\chi_S: H_SS \rightarrow G_SS$ be the map defined as in the second paragraph of the Section $1$. Then $H_S \subseteq Ker \chi_S \subseteq Core_G(H)=\{1\}$. Hence $H_S=\{1\}$, that is $S$ is a subgroup of $G$.  %Number of isomorphism classes can not be 2 (see \cite[Theorem, p. 1718]{viv1}). If $|I(\langle S \rangle,H_S)|=1$, then $H_S \trianglelefteq \langle S \rangle$ (see \cite[Main Theorem, p.643]{rjpf}). Which means $H_S \subseteq Core_GH=\left\{1\right\}$. $|I(\langle S \rangle,H_S)|=3$ shows that $[\langle S \rangle:H_S]=3$ that is, $|S|=3$ (see \cite[Theorem A, p.****]{viv2}). But in this case, $|\mathcal{I}(G,H)| \leq 3$, which is a contradiction.
\end{proof}
\begin{proposition}\label{p1'} Let $H$ be a corefree subgroup of a finite group $G$. Let $S \in \mathcal{T}(G,H)$ such that $\langle S \rangle=G$. Then $Aut_HG$ acts transitively on the set $\{T \in \mathcal{T}(G,H)|T \cong S \}$.
\end{proposition}
\begin{proof} The proof follows from the first paragraph of the proof of \cite[Proposition 2.7, p.652]{rjpf}.
\end{proof}
\begin{proposition}\label{p2} Let $(G,H)$ be a minimal counterexample. Let $N$ be a proper $Aut_HG$-invariant subgroup of $G$ containing $H$ properly. 
Then there exists $S \in \mathcal{T}(G,H)$ such that $S \neq TL$ for any $T\in \mathcal{T}(N,H)$ and $L \in \mathcal{T}(G,N)$.
\end{proposition}
\begin{proof} Assume that each $S \in \mathcal{T}(G,H)$ can be written as $S=TL$, for some $T\in \mathcal{T}(N,H)$ and $L \in \mathcal{T}(G,N)$. Then $\left|\mathcal{T}(G,H)\right| \leq \left|\mathcal{T}(N,H)\right|\left|\mathcal{T}(G,N)\right|$. Keeping the same lines as in the proof of \cite[Lemma 2.5, p.1720]{viv1}, we observe that $|H|=2, N \cong C_4$ and $|\mathcal{I}(G,N)| \leq 3$. If $|\mathcal{I}(G,N)|=1$, then $N \trianglelefteq G$ (see \cite[Main Theorem, p.643]{rjpf}). But then this implies $H \trianglelefteq G$, a contradiction. Also, by \cite[Theorem, p. 1718]{viv1} $\left|\mathcal{I}(G,N)\right| \neq 2$.

Thus $|\mathcal{I}(G,N)|=3$. Then $[G:N]=3$ (see \cite[Theorem A, p.2025]{viv2}). This means that $|G|=12$ and $G$ contains a cyclic subgroup of order 4. By the classification of non-abelian groups of order 12, the only choice for $G$ is $G \cong C_3 \rtimes C_4$. But $C_3 \rtimes C_4$ has a unique subgroup of order 2, hence normal in $G$. This is again a contradiction.
\end{proof}
\begin{corollary}\label{p2c} Let $(G,H)$ be a minimal counterexample. Let $N$ be a proper $Aut_HG$-invariant subgroup of $G$ containing $H$ properly. Let $K \in \mathcal{T}(N,H)$. Then there exists $S \in \mathcal{T}(G,H)$ containing $K$ such that $S\neq KL$ for any $L \in \mathcal{T}(G,N)$.
\end{corollary}
%An equivalence relation on a right loop $S$ is called a \textit{congruence} if it is a subright loop of $S \times S$. We recall from \cite[Definition 2.8, p. 2689]{rpsc} that an \textit{invarint subright loop} of a right loop $S$ is precisely the equivalence class of the identity of a congruence in $S$. A right loop is called \textit{simple} if it has no nontrivial proper sub right loop. 

\begin{lemma}\label{l1'} Let $G=D_8$ and $H$ be a non-normal subgroup of $G$ of order $2$. Then $\left|\mathcal{I}(G,H)\right|=6$ 
\end{lemma} 
\begin{proof} Let $H=\{1,x\}$. Let $y\in G$ be of order 4 and $N=\langle y \rangle$. Then $G=\langle x,y\rangle$ with $xyx=y^3$ and $Aut_HG=\{I,i_x\}$, where $i_x$ denotes the inner automorphism of $G$ determined by $x$. Let $\epsilon :N\rightarrow H$ be a function with $\epsilon (1)=1$. Let $S_{\epsilon}=\{\epsilon(y^i)y^i|1 \leq i \leq 4 \} \in \mathcal{T}(G,H)$. Note that $xS_{\epsilon}x^{-1}=S_{\epsilon}$ means that $\epsilon(y^i)y^i \in S_{\epsilon}$ if and only if $\epsilon(y^i)y^{4-i} \in S_{\epsilon}$. This implies that $S_1=N=\{I,y,y^2,y^3\}$, $S_2=\{I,xy,y^2,xy^3\}$, $S_3=\{I,y,xy^2,y^3\}$ and $S_4=\{I,xy,xy^2,xy^3\}$ are the fixed point of the action of $Aut_HG$ on $\mathcal{T}(G,H)$. Since $|\mathcal{T}(G,H)|=8$, there are two orbits of length 2. Let $S_5=\{I,xy,y^2,y^3\}$, $S_6=\{I,xy,xy^2,y^3\}$. Then the NRTs $S_5$ and $S_6$ are in the distinct $Aut_HG$ orbits which are not singletons.   

One observes that $S_1 \cong C_4$, $S_2 \cong C_2 \times C_2$ and $\langle S_i \rangle =G$ $(3 \leq i \leq 6)$. Then $S_1 \ncong S_2$. Further, if $S_i \cong S_j$ for $3 \leq i \neq j \leq 6$, then by Proposition \ref{p1'} $xS_ix^{-1}=S_j$, which is a contradiction (for $S_i$ and $S_j$ lie in different $Aut_HG$-orbits). 
\end{proof} 

\begin{remark} \label{r1} (i) By the same argument as above, we find that $\left|\mathcal{I}(G,H)\right|=20$, where $G=D_{12}$ and $H$ is a non-normal subgroup of $G$ of order 2 (see also \cite[Remark 2.7, p. 2028]{viv2}). In \cite{viv3}, a formula for the number of orbits of the action of $Aut_HG$ on $\mathcal{T}(G,H)$ has been obtained.
%\end{remark}

%\begin{remark} \label{r2} 
(ii) As argued in the proof of Lemma \ref{l1'}, we see that if $H$ is a corefree subgroup of a finite group $G$, then the NRTs from different orbits of the action of $Aut_HG$ on $\mathcal{T}(G,H)$ which generate the group $G$ represent pairwise non-isomorphic NRTs. 
\end{remark}

\begin{lemma}\label{l} Let $G=Alt(4)$ and $H$ be a subgroup of $G$ of order $2$. Then $\left|\mathcal{I}(G,H)\right|=5$.
\end{lemma}
\begin{proof}Since $H$ is a subgroup of $G=Alt(4)$ of order 2, there is a unique Sylow $2$-subgroup $P$ of $Sym(4)$ such that $H=Z(P)$, the center of $P$. Since $Aut(Alt(4)) \cong Inn(Sym(4)) \cong Sym(4)$, $Aut_HG \cong N_{Sym(4)}(H)(=P)$, the normalizer of $H$ in $Sym(4)$.
As there is no subgroup of $Alt(4)$ of index 2, $\langle S \rangle =Alt(4)$ for all $S \in \mathcal{T}(G,H)$. By Remark \ref{r1}(ii), $\left|\mathcal{I}(Alt(4),H)\right|$ is precisely the number of orbits of the conjugation action of $P$ on $\mathcal{T}(Alt(4),H)$.  

We may assume that $H=\{I,x=(1,2)(3,4)\}$ (as any two elements of order 2 in $G$ are conjugate). Let $P=\langle (1,2),(1,3,2,4) \rangle$. Then $Aut_H(Alt(4))= \{i_g|g \in P\}$, where for $g \in P$, $i_g$ denotes the conjugation of $Alt(4)$ by $g$. 

Let $T=\{I,y=(1,3)(2,4)\}$, $L=\{I,z=(1,2,3),z^{-1}=(1,3,2)\}$ and $S=TL$. Then $S \in \mathcal{T}(G,H)$. Let $S_1=S=\{I,y, z,z^{-1},yz^{-1},yz\}$, $S_2=\{I,y, z,z^{-1},yz^{-1},xyz\}, S_3=\{I,y, z,z^{-1},xyz^{-1},xyz\}$, $S_4=\{I,y, z,z^{-1}, \\ xyz^{-1},yz\}$ and $S_5=\{I,y, z,xz^{-1},yz^{-1},yz\}$. Let $i \in \{1,\cdots , 5 \}$. We note that, if $g \in P$ such that $gS_ig^{-1}=S_i$, then $gyg^{-1}=y$ and so $g=x$. Since $xzx^{-1}=yz$ and $xz^{-1}x^{-1}=xyz^{-1}$, it follows that $S_1,S_2,S_3,S_4$ and $S_5$ lie in different orbits with orbit length 8,8,8,4 and 4 respectively. 
\end{proof}

\begin{lemma}\label{l2}
Let $(G,H)$ be a minimal counterexample and $N$ be a proper $Aut_HG$-invariant subgroup of $G$. Let $K \in \mathcal{T}(N,H)$ which is a subgroup of $N$. Then $[G:N] \neq 2$. 
\end{lemma}
\begin{proof} If possible, assume that $[G:N]=2$. Further, assume that $[N:H]=2$. Since $Core_G(H)= \{1\}$ (Proposition \ref{p1}(i)), we can identify $G$ with a subgroup of $Sym(4)$. The only possibility for the pair $(G,N)$ we are left with is $G$, a Sylow $2$-subgroup of $Sym(4)$, $N \cong C_2 \times C_2$. Hence $G \cong D_8$. By Lemma \ref{l1'}, $|\mathcal{I}(G,H)|=6$, a contradiction. Thus $[N:H]>2$. Let $L=\{1,l\} \in \mathcal{T}(G,N)$, $k_2, k_3 (\neq k_2) \in K \setminus \{1\}$ and $U=K \setminus \{k_3\} \cup \{hk_3\}$, where $h \in H \setminus \{1\}$. We note that $U$ is not a subgroup of $N$, for $k_2,k_3k_2^{-1} \in U$ but $k_3=(k_3k_2^{-1})k_2 \notin U$. Let $S_1=KL$ and $S_2=UL$. As $U =N \cap S_2$ is not a subgroup of $N$, $S_2$ is not subgroup of $G$. Let $S_3=(S_1 \setminus \{k_3l\}) \cup \{hk_3l\}$, $S_4=(S_1 \setminus \{k_2l,k_3l\}) \cup \{hk_2l,hk_3l\}$ and $S_5=S_2 \setminus \{hk_3l\} \cup \{k_3l\}$. Observe that $S_i$ ($3 \leq i \leq 5$) are not subgroups of $G$, for $hk_3=(hk_3l)(l^{-1}) \notin S_i$ ($i=3,4$) and $k_3=(k_3l)(l^{-1}) \notin S_5$.

We also claim that $S_i \neq T^{\prime}L^{\prime}$ ($3 \leq i \leq 5$) for any $T^{\prime} \in \mathcal{T}(N,H)$ and $L^{\prime} \in \mathcal{T}(G,N)$. If possible, suppose that $S_3=T^{\prime}L^{\prime}$ for some $T^{\prime} \in \mathcal{T}(N,H)$ and $L^{\prime} \in \mathcal{T}(G,N)$. Then $T^{\prime}=S_3 \cap N=K$. Since $[G:N]=2$, either $hk_3l \in L^{\prime}$ or $kl \in L^{\prime}$ for some $k \in K \setminus \{k_3\}$. Further, $N=HK$ is a subgroup of $G$, $hk_3=k_3^{\prime}h^{\prime}$ for some $k_3^{\prime} \in K$ and $h^{\prime} \in H$. Also $h^{\prime} \neq 1$, for $K \in \mathcal{T}(N,H)$. Assume that $hk_3l \in L^{\prime}$. Then ${k_3^{\prime}}^{-1}(hk_3l)=h^{\prime}l \in KL^{\prime}=S_3$. This is a contradiction, for $l \in S_3$ and $h^{\prime} \neq 1$. Thus $kl \in L^{\prime}$ for some $k \in K \setminus \{k_3\}$. Since $hk_3l \in S_3$, we have $hk_3l=k^{\prime}(kl)$ for some $k^{\prime} \in K$. This implies that $hk_3 \in K$, a contradiction. Similarly $S_4 \neq T^{\prime}L^{\prime}$ for any $T^{\prime} \in \mathcal{T}(N,H)$ and $L^{\prime} \in \mathcal{T}(G,N)$. We again claim the same for $S_5$. If possible, suppose that $S_5=T^{\prime}L^{\prime}$ for some $T^{\prime} \in \mathcal{T}(N,H)$ and $L^{\prime} \in \mathcal{T}(G,N)$. As argued above $T^{\prime}=U$. Since $[G:N]=2$, $kl \in L^{\prime}$ for some $k \in K$. If $k=1$, then $(hk_3)l \in UL^{\prime}=S_5$, a contradiction. Thus $k \neq 1$. This means that $k_3k \in U$, which implies that $k_3kl \in S_5$. Hence $(hk_3)kl=h(k_3kl)$ can not be in $UL^{\prime}=S_5$, a contradiction.

We now show that $S_i$ ($1 \leq i \leq 5$) are pairwise non-isomorphic NRTs of $H$ in $G$. Since $S_i \neq T^{\prime}L^{\prime}$ $(3 \leq i \leq 5)$ for any $T^{\prime} \in \mathcal{T}(N,H)$ and $L^{\prime} \in \mathcal{T}(G,N)$, $S_1 \ncong S_i$ ($3 \leq i \leq 5$) and $S_2 \ncong S_i$ $(3 \leq i \leq 5)$. Further, since $U$ is not a subgroup of $N$, $S_1 \ncong S_2$, $S_3 \ncong S_5$ and $S_4 \ncong S_5$ (as $K \subseteq S_i$ for $i \in \{1,3,4\}$ and $N$ is an $Aut_HG$-invariant subgroup of $G$). Next, assume that $S_3 \cong S_4$. Then, by Proposition \ref{p1'} there exists $f \in Aut_HG$ such that $f(S_3)=S_4$. Hence $f(K)=K$ (for $N$ is an $Aut_HG$-invariant subgroup of $G$). Assume that $f(l)=kl$ for some $k \in K$. Then $f(k^{\prime}l)=f(k^{\prime})kl \in Kl$ for all $k^{\prime} \in K$. In this case, either $hk_2l$ or $hk_3l$ can not be image of any element in $S_4$, a contradiction. Hence $f(l) = hk_2l$ or $f(l)=hk_3l$. Suppose that $f(l) = hk_2l$. Since $N=HK$ is a subgroup of $G$ and $K \in \mathcal{T}(N,H)$, $hk_2=k_2^{\prime}h^{\prime}$ for some $k_2^{\prime} \in K \setminus \{1\}$ and $h^{\prime} \in H \setminus \{1\}$. Let $k_2^{\prime \prime} \in K$ such that $f(k_2^{\prime \prime})={k_2^{\prime}}^{-1}$. This implies that $f(k_2^{\prime \prime}l)={k_2^{\prime}}^{-1}(hk_2l)=h^{\prime}l$, which is a contradiction for $l \in S_3$. Similarly $f(l) \neq hk_3l$. Hence $S_3 \ncong S_4$.

%Hence $K \subseteq S_1$, $S_1 \ncong S_i$ for $i=2,5$ (as $U \subseteq S_i$ and $N$ is $Aut_HG$-invariant subgroup of $G$). Similerly $S_2 \ncong S_3$ or $S_4$ and $S_3 \ncong S_5$. Also $S_1 \ncong S_i$ for $i=3,4$, because $S_i \neq T^{\prime}L^{\prime}$ for any $T^{\prime} \in \mathcal{T}(N,H)$ and $L^{\prime} \in \mathcal{T}(G,N)$. At last, $S_3 \ncong S_4$. If $S_3 \cong S_4$, then by Proposition \ref{p1'} there exists $f \in Aut_HG$ such that $f(S_3)=S_4$. As $N$ is $Aut_HG$-invariant, $f(K)=f(S_3 \cap N)=S_4 \cap N=K$. If $f(l)=k^{\prime}l$, where $k^{\prime} \in K$, then $f(kl)=f(k)f(l)=f(k)k^{\prime}l$. But then $hk_2l$ or $hk_3l$ can not be image of any element in $S_4$, this is a contradiction. Hence $f(l) = hk_2l$ or $hk_3l$. Assume that $|L_1| \geq 3$. LKet $l_2 \in L \setminus \{1,l_1\}$ and choose $k_1 \in K$ such that $k_1u_1 \notin U$. Then $k_1, u_1l_2 \in S_5$ but $k_1(u_1l_2) \notin S_5$, acontradiction. Thus $|L_1|=2$ and $l_1=l$. But again in this case, $hk_2l$ or $hk_3l$ can not be image of any element in $S_4$, a contradiction. Therefore $S_3 \ncong S_4$.  
\end{proof}
\begin{lemma} \label{l2'} Let $G$ be a finite group. Let $H$ be a non-normal, abelian, corefree subgroup of $G$ and $N$  be a normal subgroup of $G$ containing $H$ such that $[N:H]=2$. Then

(i) if $[G:N]=2$, then $G \cong D_8$ and $|H|=2$.

(ii) if $[G:N]=3$, then $G \cong Alt(4) \times C_2$ and $H \cong C_2 \times C_2$ or $G \cong Alt(4)$ and $H \cong C_2$ .
\end{lemma}
\begin{proof} (i) Assume that $[G:N]=2$. Then as argued in the first few lines of Lemma \ref{l2}, $G \cong D_8$, $N \cong C_2 \times C_2$ and $H$ is a non-normal subgroup of $G$ of order $2$.

(ii) Assume that $[G:N]=3$. We can identify $G$ with a subgroup of $Sym(6)$. Since the order of an abelian subgroups of $Sym(6)$ is at most $9$ (\cite[Theorem 1, p. 70]{bg}), $|H| \leq 9$. Further, since $Sym(6)$ has no subgroup of order 54, $|H| \neq 9$. Assume that $|H|=8$. Then $|N|=16$. Hence $N$ is a Sylow 2-subgroup of $Sym(6)$. Since $G \subseteq N_{Sym(6)}(N)$ (the normalizer of $N$ in $Sym(6)$) and a Sylow 2-subgroup of $Sym(6)$ is self-normalizing (\cite[Corollary 1, p. 123]{lw}), $N=G$, a contradiction. Further $|H| \neq 7$, for $Sym(6)$ does not contain a 7-cycle.

Next, assume that $|H|=6$. Then $|N|=12$. By classification of groups of order 12, $N \cong D_{12}$. Since $N \trianglelefteq G$ and $N$ contains a unique cyclic subgroup of order 6, $H \trianglelefteq G$, a contradiction. Next, assume that $|H|=5$. Then $N \cong D_{10}$. Hence, in this case also $H \trianglelefteq G$. Similar argument shows that $|H| \neq 3$.

Next, assume that $|H|=4$. Then $|N|=8$. Assume that $H \cong C_4$. Since $Sym(6)$ does not contain an 8-cycle, $N \ncong C_8$. Suppose that $N \cong (C_4 \times C_2)$. Then, by \cite[84(ii), p.102]{wb} $G \cong N \times C_3$. This is a contradiction, for the order of an abelian subgroup of $Sym(6)$ is at most $9$ (\cite[Theorem 1, p. 70]{bg}). Assume that $N \cong D_8$. Then as argued in the above paragraph $H \trianglelefteq G$, a contradiction. Let $N \cong Q_8$ (the quaternion group of order 8). By \cite[84(iv), p. 103]{wb}, either $G \cong (Q_8 \times C_3)$ or $G \cong (Q_8 \rtimes C_3)$. If $G \cong (Q_8 \times C_3)$, then $H \trianglelefteq G$ (\cite[5.3.7 (Dedekind, Baer), p. 143]{rob}), a contradiction. Hence $G \cong (Q_8 \rtimes C_3)$. But in this case $|Core_G(H)|=2$, a contradiction. Thus, $H \cong C_2 \times C_2$. This implies $N \cong D_8$ or $N \cong C_2 \times C_4$ or $N \cong C_2 \times C_2 \times C_2$. Also if $N \cong D_8$, then $H \trianglelefteq G$ (since $N$ contains a unique non-cyclic subgroup of order $4$). This is a contradiction. If $N \cong C_2 \times C_4$, then $G \cong (C_2 \times C_4) \times C_3$ (\cite[84(ii), p.102]{wb}), a contradiction. Thus $N \cong C_2 \times C_2 \times C_2$. By \cite[84(iii), p.102]{wb}, $G \cong Alt(4) \times C_2$.       
 
 Lastly, assume that $|H|=2$. Then either $N \cong C_4$ or $N \cong C_2 \times C_2$. If $N \cong C_4$, then by classification of non-abelian groups of order 12, $G \cong C_4 \rtimes C_3$. But in this case $H \trianglelefteq G$, a contradiction. Hence $N \cong C_2 \times C_2$. Since $N \trianglelefteq G$, $G \cong Alt(4)$.    
\end{proof} 

\begin{lemma}\label{l1}
Let $G \cong Alt(4) \times C_2$ and $H$ be a corefree subgroup of $G$ of index $6$. Then $\left|\mathcal{I}(G,H)\right| > 4$.
\end{lemma}
\begin{proof}
Since $Core_G(H)=\{1\}$, we can identify $G$ with a subgroup of $Sym(6)$. Thus, there exist subgroups $K$ and $L$ of $Sym(6)$ such that $K \cong Alt(4)$, $L \cong C_2$, $K \cap L=\{1\}$, $G=KL$, $K \trianglelefteq G$ and $L \trianglelefteq G$. Further, since $Core_G(H)=\{1\}$, $H \cap L=\{1\}$ and $|H \cap K|=2$. Thus $HK=G$ and so $\mathcal{T}(K,K \cap H) \subseteq \mathcal{T}(G,H)$. By Lemma \ref{l}, $|\mathcal{I}(K, K \cap H)|=5$. Thus $\left|\mathcal{I}(G,H)\right| > 4$.
\end{proof}

\begin{lemma} \label{ln1} Let $(G,H)$ be a minimal counterexample. Let $N$ be an $Aut_HG$-invariant subgroup of $G$ such that $[N:H]=2$. Assume that there exists $T \in \mathcal{T}(G,H)$ which is a subgroup of $N$. Then $[G:N] \in \{2,3\}$. 
\end{lemma}
\begin{proof} If possible, assume that $[G:N] \geq 4$. Let $T=\{1, x\} \in \mathcal{T}(N,H)$ which is a subgroup of $N$. Let $L=\{1,l_2,l_3,l_4,...,l_r\} \in \mathcal{T}(G,N)$ and $h \in H \setminus \{1\}$.
Consider $S=TL$, $S_1=(S \setminus \{x\}) \cup \{hx\}$, $S_2=(S \setminus \{l_r\}) \cup \{hl_r\}$, $S_3=(S \setminus \{l_{r-1},l_r\}) \cup \{hl_{r-1},hl_r\}$ and $S_4=(S \setminus \{l_{r-2},l_{r-1},l_r\}) \cup \{hl_{r-2},hl_{r-1},hl_r\}$. As argued in the proof of \cite[Lemma 2.12, p.2030]{viv2}, $S$ and $S_i$ $(1 \leq i \leq 3)$ are non-isomorphic NRTs.

As argued in the second paragraph of the proof of \cite[Lemma 2.12, p.2030]{viv2}, we can show that $S_4$ is not a subgroup of $G$ and $S_4 \neq T^{\prime}L^{\prime}$ for any $T^{\prime} \in \mathcal{T}(N,H)$ and $L^{\prime} \in \mathcal{T}(G,N)$. This shows that $S_4 \ncong S$. Now, we show that $S_4 \ncong S_i$ $(1 \leq i \leq 3)$. Assume that $S_4 \cong S_3$. Then by Proposition \ref{p1'}, there exists $f \in Aut_HG$ such that $f(S_3)=S_4$. Since $N$ and $H$ are $Aut_HG$-invariant, $f(x)=x$. Further, since $[G:N] \geq 4$, there exist $i \in \{0,1 \}$ and $k \in \{2 \cdots r-3 \}$ such that $f(x^il_k)=hl_j$ for some $j \in \{r-2,r-1,r\}$. Hence $f(x^{i+1}l_k)=f(x)hl_j=xhl_j \notin S_4$, a contradiction. Similarly, $S_4 \ncong S_i$ $(i=1,2)$. Thus $|\mathcal{I}(G,H)|>4$, a contradiction.       
\end{proof}
\begin{proposition} \label{prn1} Let $(G,H)$ be a minimal counterexample. Then $\langle S \rangle = G$ for all $S \in \mathcal{T}(G,H)$.
\end{proposition}
\begin{proof}
On the contrary, assume that there exists $ S^{\prime} \in \mathcal{T}(G,H)$ such that $\langle S^{\prime} \rangle \neq G$. Then $S^{\prime}$ is a subgroup of $G$ (Proposition \ref{p1}(iii)). Thus $K=S^{\prime} \cap N \in \mathcal{T}(N,H)$ is a subgroup of $N$. Further, assume that all the members of $\mathcal{T}(N,H)$ are subgroups of $N$. This implies that $N \cong H \rtimes C_2$ and $H$ is abelian (\cite[Lemma 2.4, p.1719]{viv1}). By Lemma \ref{ln1}, $[G:N]=2$ or $[G:N]=3$.

Assume that $[G:N]=2$. By Lemma \ref{l2'}(i), $G \cong D_8$ and $H \cong C_2$. But in this case, $|I(G,H)|=6$ (Lemma \ref{l1'}), a contradiction.
Thus $[G:N]=3$. Hence $G$ is isomorphic to a subgroup of $Sym(6)$. By Lemma \ref{l2'}(ii), the choices for the pair $(G,H)$ in this case are $G \cong Alt(4) \times C_2$, $H \cong C_2\times C_2$ or $G \cong Alt(4)$, $H \cong C_2$. By Lemmas \ref{l1} and \ref{l}, $|\mathcal{I}(G,H)| >4$, again a contradiction. Thus, there exists $U \in \mathcal{T}(N,H)$ which is not a subgroup of $G$. 

Let $L_1 \in \mathcal{T}(G,N)$, $S_1=S^{\prime}$ and $S_2=UL_1$. By Corollary \ref{p2c}, there exists $S_3 \in \mathcal{T}(G,H)$ such that $U \subseteq S_3$ and $S_3 \neq UL$ for any $L \in \mathcal{T}(N,H)$. Also, let $S_4=KL_1$, $S_4^{\prime}=(S_4 \setminus\{l\}) \cup \{hl\}$, where $h \in H \setminus \{1\}$ and $l \in L_1 \setminus \{1\}$. Let $S_5=(UL_1 \setminus U) \cup K$. Let ${S_4}^{\epsilon}$ denote $S_4$ if it is not subgroup of $G$, otherwise it is $S_4^{\prime}$. As argued in the paragraphs three and four of the proof of \cite[Lemma 2.16, p.2032]{viv2}, $S_1$, $S_2$, $S_3$ and $S_4^{\epsilon}$ are pairwise non-isomorphic NRTs and each of $S_2$, $S_3$ and $S_4^{\epsilon}$ generates $G$. We show that $S_5$ is not isomorphic to $S_i$ $(1 \leq i \leq 4)$.

The NRT $S_5$ is not a subgroup of $G$, for if $u \in U \setminus K$, then $l,ul \in S_5$, but $u=(ul)l^{-1} \notin S_5$.  

Next, assume that $S_5 = K^{\prime}L^{\prime}$, for some $K^{\prime} \in \mathcal{T}(N,H)$ and $L^{\prime} \in \mathcal{T}(G,N)$. Then $K^{\prime}=N \cap S_5=K$. Fix $u \in U \setminus K$. If possible, assume that $ku \in U$ for all $k \in K \setminus U$. Then $ku \in U \setminus K$ for all $k \in K \setminus U$ ( since $K$ is a subgroup of $G$). Thus the map $k \mapsto ku$ is a bijection from $K \setminus U$ to $U \setminus K$. But, this is a contradiction, for $u \in U\setminus K$ is not an image under this map. Thus, there exists $k \in K \setminus U$ such that $ku \notin U$. Fix such a $k \in K \setminus U$.

Since $k^{\prime}l, u^{\prime}l \in S_5$ ($k^{\prime} \in K \cap U, u^{\prime} \in U \setminus K$) are in the same right coset of $N$ in $G$ and $L^{\prime} \subseteq S_5$, either $k^{\prime}l \in L^{\prime}$ for some $k^{\prime} \in K \cap U$ or $u^{\prime}l \in L^{\prime}$ for some $u^{\prime} \in U \setminus K$. Suppose that $u^{\prime}l \in L^{\prime}$ for some $u^{\prime} \in U \setminus K$. Then as argued in the above paragraph, there exists $k^{\prime} \in K \setminus U$ such that $k^{\prime}u^{\prime} \notin U$. This implies that $k^{\prime}(u^{\prime}l) \notin S_5$, which is a cotntradiction. Thus $k^{\prime}l \in L^{\prime}$ for some $k^{\prime} \in K \cap U$. But in this case, for any $u^{\prime} \in U \setminus K$, $u^{\prime}l \in S_5$ can not be written as a product of a member of $K$ and a member of $L^{\prime}$. This is again a contradiction. Thus $S_5 \neq K^{\prime}L^{\prime}$ for any $K^{\prime} \in \mathcal{T}(N,H)$ and for any $L^{\prime} \in \mathcal{T}(G,N)$. Hence $S_5$ is neither isomorphic to $S_2$ nor isomorphic to $S_4$ (if $S_4^{\epsilon}=S_4$). If possible assume that $S_5 \cong S_3$. Then by Proposition \ref{p1'} there exists $f \in Aut_HG$ such that $f(S_5) =S_3$. Since $N$ is an $Aut_HG$-invariant subgroup of $G$, $f(K)=U$. This is a contradiction (for $U$ is not a subgroup of $G$).

Finally, assume that $S_4^{\epsilon}=S_4^{\prime}$ and $S_5 \cong S_4^{\prime}$. Since $\langle S_4^{\prime} \rangle = \langle S_5 \rangle=G$, by Proposition \ref{p1'} there exists $f \in Aut_HG$ such that $f(S_4^{\prime})=S_5$. As $N$ is $Aut_HG$-invariant, $f(K)=f(S_4^{\prime} \cap N)=S_5 \cap N=K$. Since $S_4$ is a subgroup of $G$, $S_5^{\prime}=(S_5 \setminus \{f(hl)\}) \cup \{f(l)\}$ is also a subgroup of $G$. We claim that $f(hl) \neq l$. Suppose that $f(hl)=l$. Then $f(l)=h_1l$, where $h_1=f(h)^{-1} \in H$. Since $k, ul \in S_5^{\prime}$, but $k(ul) \notin S_5^{\prime}$, a contradiction (for $S_5^{\prime}$ is a subgroup of $G$). Thus $f(hl) \neq l$. Since $f(hl) \notin K$, there exists $u_1 \in U$ and $l_1 \in L_1$ such that $f(hl)=u_1l_1$. By Lemma \ref{l2}, $|L_1| \geq3$.  Let $l_2 \in L_1 \setminus \{1,l_1\}$. Then $k,ul_2 \in S_5^{\prime}$, but $k(ul_2) \notin S_5^{\prime}$, a contradiction.

\end{proof}
\begin{proposition} \label{prn2} Let $(G,H)$ be a minimal counterexample and $S \in \mathcal{T}(G,H)$. Then $S$ is indecomposable.
\end{proposition}
\begin{proof} If possible, suppose that $S$ is decomposable and $S=S_1 \times S_2 \times \cdots \times S_n$ ($n \geq 2$) is a Remak-Krull-Schimdt decomposition of $S$ (see \cite[Theorem 1.11, p.648]{rjpf}). By Proposition \ref{prn1} and \cite[Remark 2.4, p. 650]{rjpf}, we may identify $(G,H)$ with $(G_{S_1}S_1 \times G_{S_2}S_2 \times \cdots \times G_{S_n}S_n,G_{S_1} \times G_{S_2} \times \cdots \times G_{S_n})$. We claim that $|\mathcal{I}(G_{S_i}S_i,G_{S_i})| \leq 4$ $(1 \leq i \leq n)$. If possible, assume that there exists $k$ $(1 \leq k \leq n)$ such that $|\mathcal{I}(G_{S_k}S_k,G_{S_k})| > 4$. Let $T_1, T_2, T_3, T_4, T_5$ $\in \mathcal{T}(G_{S_k}S_k,G_{S_k})$ be pairwise non-isomorphic NRTs. Then $L_i=S_1 \times \cdots \times S_{k-1} \times T_i \times S_{k+1}\times \cdots \times S_n$ $(1 \leq i \leq 5)$ are pairwise non-isomorphic NRTs of $H$ in $G$ by Remak-Krull-Schimdt Theorem (see \cite[Theorem 1.11, p.648]{rjpf}), a contradiction.

Since $(G,H)$ is a minimal counterexample, $|\mathcal{I}(G_{S_k}S_k,G_{S_k})| \leq 3$ for all $k \in \{1 \cdots n\}$. Further, since by \cite[Theorem, p. 1718]{viv1} $|\mathcal{I}(G_{S_k}S_k,G_{S_k})| \neq 2$ for all $k \in \{1 \cdots n\}$, either $|\mathcal{I}(G_{S_k}S_k,G_{S_k})| =1$ or $|\mathcal{I}(G_{S_k}S_k,G_{S_k})|=3$. In either case, we get $T_k \in \mathcal{T}(G_{S_k}S_k,G_{S_k})$ which is a group. Let $T=T_1 \times T_2 \times \cdots \times T_n$. Then $T \in \mathcal{T}(G,H)$ and is a group. Since $\langle T \rangle=G$ (Proposition \ref{prn1}), by Proposition \ref{p1}(ii), $H=G_T=\{1\}$, a contradiction. 
\end{proof}
\section{Proof of the Theorem}
In this section, we study some more properties of a minimal counterexample and reduce it to the case of a finite non-abelian simple group. Then we apply the classification of finite simple groups (the knowledge of the order of automorphism groups of finite non-abelian simple groups) to complete the proof of the Main Theorem.
\begin{proposition}\label{prn3} Let $(G,H)$ be a minimal counterexample. Then $G$ is indecomposable.
\end{proposition}
\begin{proof} If possible, suppose that $G$ is decomposable. Let $G_1$ and $G_2$ be nontrivial proper normal subgroups of $G$ such that $G=G_1G_2$ and $G_1 \cap G_2=\{1\}$. Let $\pi_i:G \rightarrow G_i$ $(i=1,2)$ be projections. Let $\pi_i(H)=U_i$ $(i=1,2)$. The restriction $\pi_i|_H$ of $\pi_i$ to $H$ induces isomorphism $\sigma_i:H/(H \cap G_1)(H \cap G_2)\rightarrow (U_i/(H \cap G_i))$ $(i=1,2)$. This gives an isomorphism $\theta=\sigma_2 \circ \sigma_1^{-1}$ from $U_1/(H \cap G_1)$ to $U_2/(H \cap G_2)$ given by $\theta(\pi_1(h)(H \cap G_1))=\pi_2(h)(H \cap G_2)$, $h \in H$. Also

\[H=\{u_1u_2 \in U_1U_2|\theta(u_1(H \cap G_1))=u_2(H \cap G_2)\}.\]

\noindent Since $Core_G(H)=\{1\}$ (Proposition \ref{p1}(i)), $H \cap G_i \neq G_i$ for $i=1,2$. Suppose that $H \cap G_1=U_1$. Then the isomorphism $\theta$ implies $H \cap G_2=U_2$. Now as argued in the second paragraph of the proof of \cite[Proposition 2.6, p. 650]{rjpf}, replacing \cite[Proposition 2.5]{rjpf} by Proposition \ref{prn2}, we get an $S \in \mathcal{T}(G,H)$ which is a direct product of nontrivial right loops $S_1$ and $S_2$. This is a contradiction (Proposition \ref{prn2}).

We may now assume that $H \cap G_i \neq U_i$ $(i=1,2)$. Suppose that $U_1=G_1$ and $U_2=G_2$. Then as argued in the third paragraph of the proof of \cite[Proposition 2.6, p. 650]{rjpf} replacing \cite[Corollary 2.3]{rjpf} by Proposition \ref{p1}(i), we get $G_2 \in \mathcal{T}(G,H)$, a contradiction (Proposition \ref{prn1}).

Thus, we may assume that $H \cap G_i \neq U_i$ $(i=1,2)$ and $U_1 \neq G_1$. Then by the same argument as in the last paragraph of the proof of \cite[Proposition 2.6, p. 650]{rjpf}, we get an $S \in \mathcal{T}(G,H)$ which is decomposable. This is a contradiction (Proposition \ref{prn2}).      
\end{proof}
\begin{lemma}\label{p4}
Let $(G,H)$ be a minimal counterexample. Let $N$ be an $Aut_HG$-invariant proper subgroup of $G$ containing $H$ properly. Then

(i) $Aut_HG$ has at most two orbits in $\mathcal{T}(N,H)$,

(ii)$H \trianglelefteq N$ and 

(iii) there are no $K_1, K_2 \in \mathcal{T}(N,H)$ which are in distinct $Aut_HG$-orbits such that $K_1$ is a subgroup of  $N$ but $K_2$ is not a subgroup of $N$.
\end{lemma}
\begin{proof} (i) Assume that $Aut_HG$ has more than two orbits in $\mathcal{T}(N,H)$. Let $K_2,~K_2$ and $K_3$ be in distinct $Aut_HG$-orbits in $\mathcal{T}(N,H)$. Fix $L \in \mathcal{T}(G,N)$. Then $K_1L$, $K_2L$, $K_3L \in \mathcal{T}(G,H)$. By Corollary \ref{p2c}, there exist $S_1$, $S_2$ and $S_3$ in $\mathcal{T}(G,H)$ containing $K_1$, $K_2$ and $K_3$ respectively such that $S_i \neq K_iL^{\prime}$, $(1 \leq i \leq 3)$ for any $L^{\prime} \in \mathcal{T}(G,N)$.

Since there is no $f \in Aut_HG$ such that $f(K_i)=K_j$ $1 \leq i \neq j \leq 3$ and $ \langle S \rangle=G$ for all $S \in \mathcal{T}(G,H)$ (Proposition \ref{prn1}), thus $S_1$, $S_2$, $S_3$, $K_1L$, $K_2L$ and $K_3L$ are pairwise non-isomorphic NRTs in $\mathcal{T}(G,H)$, a contradiction. 

(ii) Suppose that $H$ is not normal in $N$. Then by \cite[Main Theorem, p.643]{rjpf} and \cite[Theorem, p. 1718]{viv1}, $|\mathcal{I}(N,H)| >2$. Let $K_1$, $K_2$ and $K_3$ be pairwise non-isomorphic NRTs in $\mathcal{T}(N,H)$. As argued in (i), we get $|\mathcal{I}(G,H)| >4$, a contradiction. Thus $H \trianglelefteq N$.

(iii) Let $K_1,K_2 \in \mathcal{T}(N,H)$ be in distinct $Aut_HG$-orbits. If possible, suppose that $K_1$ is a subgroup of $N$ but $K_2$ is not a subgroup of $N$. Let $L=\{1,l_2, \cdots, l_r\} \in \mathcal{T}(G,N)$. By Lemma \ref{l2}, $|L| \geq 3$. Let $h \in H \setminus \{1\}$. Let $S_1=K_1L$, $S_2=K_2L$, $S_3=(S_1 \setminus \{l_r\}) \cup \{hl_r\}$, $S_4=(S_1 \setminus \{l_{r-1},l_r\}) \cup \{hl_{r-1},hl_r\}$. By Corollary \ref{p2c}, there exists $S_5 \in \mathcal{T}(G,H)$ containing $K_2$ such that $S_5 \neq K_2L^{\prime}$ for any $L^{\prime} \in \mathcal{T}(G,N)$. By Proposition \ref{prn1}, $\langle S_i \rangle= G$ for all $i$ $(1 \leq i \leq 5)$. 

We claim that $S_i$ ($1 \leq i \leq 5$) are pairwise non-isomorphic NRTs in $\mathcal{T}(G,H)$. Since $N$ is an $Aut_HG$-invariant subgroup of $G$, $S_i \ncong S_j$ ($i=1,3,4$; $j=2,5$) (for otherwise by Proposition \ref{p1'} $f(K_1)=K_2$ for some $f \in Aut_HG$ ). Assume that $S_1 \cong S_3$. By Proposition \ref{p1'} there exists $f \in Aut_HG$ such that $f(S_1)=S_3$. As $N$ is an $Aut_HG$-invariant subgroup of $G$, $f(K_1)=K_1$. If possible, suppose that $f(l)=hl_r$ for some $l \in L$. Let $k \in K_1 \setminus \{1\}$ and $f(k)=k_1$. Then $f(kl)=k_1(hl_r)$. Since $N=HK_1$ is a subgroup of $G$ and $K_1 \in \mathcal{T}(N,H)$, $k_1h=h^{\prime}k_1^{\prime}$ for some $k_1^{\prime} \in K_1 \setminus \{1\}$ and $h^{\prime} \in H \setminus \{1\}$. This implies that $f(kl)=h^{\prime}(k_1^{\prime}l_r) \in S_3$. This is a contradiction, for $k_1^{\prime}l_r \in S_3$. Therefore $f(k^{\prime}l)=hl_r$ for some $k^{\prime} \in K_1 \setminus \{1\}$. This implies that $f(l)=f(k^{\prime})^{-1}hl_r \in S_3$. As argued above, this gives a contradiction. Hence $S_1 \ncong S_3$. Similar arguments prove that $S_1 \ncong S_4$ and $S_3 \ncong S_4$. Now, assume that $S_2 \cong S_5$. By Proposition \ref{p1'} there exists $f \in Aut_HG$ such that $f(S_2)=S_5$. Since $N$ is $Aut_HG$-invariant subgroup of $G$, $f(K_2)=K_2$ and so $S_5=K_2f(L)$, a contradiction to the choice of $S_5$.  
%$S_i \neq K^{\prime}L^{\prime}$ ($i=3,4$) for any $K^{\prime} \in \mathcal{T}(N,H)$ and $L^{\prime} \in \mathcal{T}(G,N)$. Assume that $S_3 = K^{\prime}L^{\prime}$ for some $K^{\prime} \in \mathcal{T}(N,H)$ and $L^{\prime} \in \mathcal{T}(G,N)$. As $N$ is $Aut_HG$-invariant subgroup of $G$, $K^{\prime}=K_1$. By the same arguement in the first paragraph of the proof of Lemma, $hl_r \notin L^{\prime}$. Hence $kl \in L^{\prime}$ for $k \in K_1$ and $l \in L \setminus \{l_r\}$. Since $hl_r \in S_3$, then $hl_r=k^{\prime}(kl^{\prime})=(k^{\prime}k)l^{prime}$ for some $k, k^{\prime} \in K_1$ and $l^{\prime} \in L \setminus \{l_r\}$.    

\end{proof}

\begin{proposition}\label{p5}
 Let $(G,H)$ be a minimal counterexample. Then $G$ is characteristically simple.
 \end{proposition}
 \begin{proof}
 
Suppose that $U$ is a nontrivial proper characteristic subgroup of $G$. 
Further, assume that $UH=G$. 
Then since $\mathcal{T}(U, U\cap H) \subseteq \mathcal{T}(UH,H)=\mathcal{T}(G,H)$ and $|U|<|G|$, by minimality of the pair $(G,H)$, $|I(U,U\cap H)|\leq 3$. 
If $|\mathcal{I}(U,U\cap H)|$=3, then $[G:H]=[U:U \cap H]=3$ (see \cite[Theorem A, p.2025]{viv2}). But in this case $|\mathcal{I}(G,H)| = 3$, a contradiction. Therefore, $|\mathcal{I}(U,U\cap H)|=1$ (Since $|\mathcal{I}(U,U\cap H)| \neq 2$ by \cite[Theorem, p. 1718]{viv1}). Also by Proposition \ref{prn1}, $\langle S \rangle= G$ for all $S\in \mathcal{T}(G,H)$. Following the steps of the first paragraph of the proof of \cite[Proposition 2.8, p. 1721]{viv1}, replacing \cite[Proposition 2.1, p. 1718]{viv1} by Proposition \ref{p1}(i), we find that $G\neq UH$, a contradiction to our assumption. Thus $UH \neq G$.

Assume that $H \nsubseteq U$. Then $\mathcal{T}(U,U \cap H) \subsetneqq \mathcal{T}(UH,H)$. Since $U$ and $H$ are $Aut_HG$-invariant, by Lemma \ref{p4}(i), $\mathcal{T}(U,U \cap H)$ and $\mathcal{T}(UH,H) \setminus \mathcal{T}(U,U \cap H)$ are $Aut_HG$-orbits. Let $T_1 \in \mathcal{T}(U,U \cap H)$ and $T_2 \in \mathcal{T}(UH,H) \setminus \mathcal{T}(U,U \cap H)$. By Lemma \ref{p4}(iii), either both $T_1$ and $T_2$ are subgroups of $G$ or both are not subgroups of $G$. 

Assume that both $T_1$ and $T_2$ are subgroups of $G$. Thus each member of $\mathcal{T}(UH,H)$ is a subgroup of $UH$. Hence by \cite[Lemma 2.4, p. 1719]{viv1}, $[UH:H]=2$. 
%Let $N$ be the smallest characteristic subgroup of $G$ containing $H$. Then $N$ is direct preduct of isomorphic simple groups (\cite[3.3.15]{rob}). Thus there exists $K \in \mathcal{T}(N,H)$ which is a subgroup of $N$. By Lemma \ref{p4}(i),(iii), all $K \in \mathcal{T}(N,H)$ are subgroups of $N$. Thus by \cite[Lemma 2.4, p.1719]{viv1}, $[N:H]=2$. Thus $N$ is an elementary abelian 2-subgroup of $G$ containg $H$. 
By Lemma \ref{ln1}, $[G:UH]=2$ or $[G:UH]=3$.

Assume that $[G:UH]=2$. Then $[G:H]=4$. As $Core_G(H)=\left\{1\right\}$ (Proposition \ref{p1}(i)), $G$ is isomorphic to a subgroup of $Sym(4)$. By Lemma \ref{l2'}(i), $G\cong D_8$ and $\left|H\right|=2$. But in this case $|\mathcal{I}(G,H)|=6$ (Lemma \ref{l1'}), a contradiction.

Thus $[G:UH]=3$. Then by Lemma \ref{l2'}(ii), $G \cong Alt(4)$ and $\left|H \right|=2$ or $G \cong Alt(4) \times C_2$ and $H \cong C_2 \times C_2$. But for both choices $|\mathcal{I}(G,H)|>4$ (Lemma \ref{l} and Lemma \ref{l1}), a contradiction.

Thus $T_1$ and $T_2$ are not subgroups of $G$. Further, assume that $[U:U \cap H]>2$. Let $h \in H \setminus U$. Let $u_2,u_3 \in T_1$ be distinct nontrivial elements. Let $T_2^{\prime}=(T_1 \setminus \{u_2\}) \cup \{hu_2\}$ and $T_3=(T_1 \setminus \{u_2,u_3\}) \cup \{hu_2,hu_3\}$. Then $T_2^{\prime}$ and $T_3$ are in same $Aut_HG$-orbit. This is a contradiction for $H$ and $U$ are $Aut_HG$-invariant and $h \notin U$. Thus $[U:U \cap H]=2$. If possible, assume that $[G:HU]=2$. Then $G$ will be isomorphic to a subgroup of $Sym(4)$. The only possibility we have in this case is $G \cong D_8$ and $H$ a non-normal subgroup of order $2$. But, then $|\mathcal{I}(G,H)|=6$ (Lemma \ref{l1'}). This is a contradiction. Thus $[G:HU]>2$.

Let $h \in H \setminus U$. Let $T=\{1,u\} \in \mathcal{T}(U,U \cap H)$ and  $L=\{1,l_2, \cdots,l_{r-1},l_r\}$\\$ \in \mathcal{T}(G,HU)$. Then $T_1=\{1,hu\}  \in \mathcal{T}(UH,H) \setminus \mathcal{T}(U,U \cap H)$. Consider $S_1=TL$, $S_2=T_1L$, $S_3=(S_2 \setminus (L \setminus \{1\})) \cup \{hl|l \in L \setminus \{1\}\}$, $S_4=(S_1 \setminus \{l_r\}) \cup \{hl_r\}$ and $S_5=(S_1 \setminus \{l_{r-1},l_r\}) \cup \{hl_{r-1},hl_r\}$. We claim that $S_i \neq T^{\prime}L^{\prime}$ ($3 \leq i \leq 5$) for any $T^{\prime} \in \mathcal{T}(UH,H)$ and $L^{\prime} \in \mathcal{T}(G,UH)$. If possible, suppose that $S_3 = T^{\prime}L^{\prime}$ for some $T^{\prime} \in \mathcal{T}(UH,H)$ and $L^{\prime} \in \mathcal{T}(G,UH)$. Then  $T_1=S_3 \cap UH=T^{\prime}$. Since $hl_r,hul_r \in S_3$ are in the same right coset of $HU$ in $G$, therefore $hl_r \in L^{\prime}$ or $hul_r \in L^{\prime}$. Assume that $hl_r \in L^{\prime}$. Then $huhl_r \in S_3$. Since $HU$ is a subgroup of $G$ and $[U:U \cap H]=2$, $uh=h^{\prime}u$, for some $h^{\prime} \in H \setminus U$. This implies that $huhl_r=hh^{\prime}ul_r \in S_3$, a contradiction. Thus $hul_r \in L^{\prime}$. Then $hu(hul_r) \in S_3$, a contradiction (for $huhu \neq h$ and $huhu \neq hu$). 

Now, assume that $S_4 = T^{\prime}L^{\prime}$ for some $T^{\prime} \in \mathcal{T}(UH, H)$ and $L^{\prime} \in \mathcal{T}(G,UH)$. Then  $T=S_4 \cap UH=T^{\prime}$. Since $hl_r,ul_r \in S_4$ are in the same right coset of $HU$ in $G$, therefore $hl_r \in L^{\prime}$ or $ul_r \in L^{\prime}$. Assume that $hl_r \in L^{\prime}$. Then $u(hl_r) \in S_4$. As argued in the above paragraph, $u(hl_r)=h^{\prime}ul_r \in S_4$ for some $h^{\prime} \in H \setminus U$, a contradiction. Thus $ul_r \in L^{\prime}$. Then $u^2l_r \in S_4$, which is again a contradiction (for $[U:U \cap H]=2$, $u^2 \in U \cap H$). Similarly, $S_5 \neq T^{\prime}L^{\prime}$ for any $T^{\prime} \in \mathcal{T}(UH,H)$ and $L^{\prime} \in \mathcal{T}(G,UH)$.

We now claim that $S_i$ $(1 \leq i \leq 5)$ are pairwise non-isomorphic NRTs in $\mathcal{T}(G,H)$. Since $S_k \neq T^{\prime}L^{\prime}$ ($3 \leq k \leq 5$) for any $T^{\prime} \in \mathcal{T}(UH,H)$ and $L^{\prime} \in \mathcal{T}(G,UH)$, $S_j \ncong S_k$ for $1 \leq j \leq 2$, $3 \leq k \leq 5$. Assume that $S_1 \cong S_2$. By Proposition \ref{p1'}, there exists $f \in Aut_HG$ such that $f(S_1)=S_2$. Since $U$ and $H$ are $Aut_HG$-invariant subgroups of $G$, $hu=f(u) \in U$, a contradiction. Thus, $S_1 \ncong S_2$. Similarly, $S_3 \ncong S_4$ and $S_3 \ncong S_5$.

Next, assume that $S_4 \cong S_5$. Then there exists $f \in Aut_HG$ such that $f(S_4)=S_5$. Since $U$ is an $Aut_HG$-invariant subgroup of $G$, $f(u)=u$. Now there exist $i \in \{0,1 \}$ and $k \in \{2, \cdots ,r-1 \}$ such that $f(u^il_k)=hl_j$ for some $j \in \{r-1,r\}$. Assume that $i=0$. Then as argued in the previous to the last above paragraph, there exists $h^{\prime} \in H \setminus U$ such that $f(ul_k)=uhl_j=h^{\prime}ul_j \in S_5$, a contradiction. Therefore, $f(ul_k)=hl_j$ for some $j \in \{r-1,r\}$. Then again there exists $h^{\prime} \in H \setminus U$ such that $f(l_k)=u^{-1}hl_j=h^{\prime}ul_j \in S_5$, a contradiction. Hence $S_4 \ncong S_5$. Thus each nontrivial characteristic subgroup $U$ contains $H$. 

%Next, assume that $S_3 \cong S_4$. By Proposition \ref{p1'}, there exists $f \in Aut_HG$ such that $f(S_3)=S_4$. Since $U$ is an $Aut_HG$-invariant subgroup of $G$, $f(u)=u$. Let $l \in L \setminus \{1\}$. If $f(l) \in L \setminus \{1,l_r\}$, then $uf(l),f(hul)=f(h)uf(l) \in S_4$. This is a contradiction (for $f(h) \in H \setminus U$). 
%Also if $f(l)=hl_r$, then $f(hul)=f(h)uhl_r$. This means that $f(h)uh=u$. Now, choose $l^{\prime} \in L \setminus \{1\}$ and $l^{\prime \prime} \in L \setminus \{1, l_r\}$ such that $f(hul^{\prime})=l^{\prime \prime}$. This implies that $f(l^{\prime})=hu^{-1}l^{\prime \prime}$, a contradiction. Therefore, $f(l)=ul$  
%Thus, there exists $l \in L \setminus \{1\}$ and $j (2 \leq j \leq r-1)$ such that $f(l)=ul_j$. Hence $f(hul)=f(h)uul_j=f(h)u^2l_j \in S_4$, a contradiction (for $u^2 \in U \cap H$ and $f(h) \in H \setminus U$). Thus $S_3 \ncong S_4$. Finally assume that $S_3 \cong S_5$. Further assume that $r>3$. Then as argued in the previous paragraph, there exists $l \in L \setminus \{1\}$ and $j$ $ (2 \leq j \leq r-2)$ such that $f(l)=ul_j$. This gives a contradiction as argued in the previous paragraph. Thus $r=3$ and $f(l)=hl_j$ $(2 \leq j \leq 3)$. But then $f(hul)=f(h)ul_j \in S_5$, a contradiction. Thus each nontrivial characteristic subgroup $U$ contains $H$.

Let $N$ be the smallest characteristic subgroup of $G$ containing $H$. Then $N$ is a direct product of isomorphic simple groups (\cite[3.3.15]{rob}). Thus there exists $K \in \mathcal{T}(N,H)$ which is a subgroup of $N$. By Lemma \ref{p4}(i),(iii) all $K \in \mathcal{T}(N,H)$ are subgroups of $N$. Thus by \cite[Lemma 2.4, p.1719]{viv1}, $[N:H]=2$. Hence $N$ is an elementary abelian 2-subgroup of $G$ containing $H$. 
By Lemma \ref{ln1}, $[G:N]=2$ or $[G:N]=3$. Now, as argued in the third and fourth paragraphs, we get a contradiction.
%%%%%%
%Then $[G:H]=6$. Since $Core_HG=\{1\}$, $G$ is isomorphic to a subgroup of $Sym(6)$. From \cite[Theorem 1.6.19, p.41]{rob}, we obtain that the sylow 2-subgroup $P$ of $Sym(6)$ is isomorphic to $(C_2 \times C_2 \times C_2) \rtimes C_2$. Since $N$ is elementary abelian 2-subgroup of $G$, $N \subseteq P$ and $[N:H]=2$, the choices of $(|G|,|H|)=(24,4)$ or $(12,2)$, (for $[G:N]=3$). If $N \cong C_2 \times C_2$, then the only choice for the pair is $G \cong Alt(4)$ and $\left|H \right|=2$ but in this case $|\mathcal{I}(G,H)|=5$ by Lemma \ref{l}.
%If $N \cong C_2 \times C_2 \times C_2$, then by \cite[84(iii), p.102]{wb}, $G \cong Alt(4) \times C_2$. Since $[N:H]=2$ and $Core_GH=\{1\}$, then $H \cong C_2 \times C_2$. By Lemma \ref{l1} $|\mathcal{I}(G,H)| >4$, a contradiction. Thus $[G:N] \geq 4$.

%Let $T=\{1, x\} \in \mathcal{T}(N,H)$ and $L=\{1,l_2,l_3,l_4,...,l_r\} \in \mathcal{T}(G,N)$.

%Consider $S=TL$, $S_1=TL \setminus \{x\} \cup \{hx\}$, $S_2=TL \setminus \{l_r\} \cup \{hl_r\}$, $S_3=TL \setminus \{l_{r-1},l_r\} \cup \{hl_{r-1},hl_r\}$ and $S_4=TL \setminus \{l_{r-2},l_{r-1},l_r\} \cup \{hl_{r-2},hl_{r-1},hl_r\}$. As argued in the proof of \cite[Lemma 2.12, p.2030]{viv2}, these are non-isomorphic NRTs, a contradiction.  
\end{proof}

\begin{corollary} \label{p5c}
Let $(G,H)$ be a minimal counterexample. Then $G$ is simple.
\end{corollary}

\begin{proof}
Since $G$ is indecomposable (Proposition \ref{prn2}) and characteristically simple (Proposition \ref{p5}), by \cite[3.3.15]{rob} $G$ is a simple group.
\end{proof}

\begin{proposition}\label{p10}
Let $(G,H)$ be a minimal counterexample. Let $S \in \mathcal{T}(G,H)$. Let $\mathcal{A}=\{ L\in \mathcal{T}(G,H)|L\cong S\}$.
Then $|\mathcal{A}| < \frac{m^{n-1}}{4}$, where $m$ and $n$ are the
order and the index of $H$ in $G$ respectively.
\end{proposition}

\begin{proof} By Proposition \ref{prn1}, $\langle S \rangle = G$. Now the proof follows from the proof of \cite[Proposition 3.4, p. 2035]{viv2}. 
%Suppose that $|\mathcal{A}| \geq \frac{m^{n-1}}{8}$. By Proposition \ref{p5c}, $G$ is a non-abelian simple group. Since $\langle S \rangle =G$, by Proposition \ref{p1} , $G  \cong G_SS \subseteq Sym(S\setminus \{1\})S $ and hence $|G| \leq n!$. Since a non-abelian simple group has order at least $60$, $n \geq 5$. By Proposition \ref{p1'}, $Aut_HG$ acts transitively on $\mathcal{A}$. So,
%\[ |Aut G| \geq |Aut_HG|\geq |\mathcal{A}|\geq \frac{m^{n-1}}{8} \]
%We show that $\frac{m^{n-1}}{2} > m^2n^2=|G|^2$. If $n=5$, then $G$ is isomorphic to a subgroup of $Sym(5)$. By Proposition \ref{p5c} $G \cong Alt(5)$. Hence $H \cong Alt(4) \subseteq Alt(5)$. In this case $Aut G  \cong Sym(5)$ and $Aut_HG \cong S_4$. Since all $S \in \mathcal{I}(Alt(5),Alt(4))$, $|\mathcal{I}(Alt(5),A_4)| \geq \frac{|\mathcal {T}(Alt(5),Aly(4))|}{|S_4|}=864$. 

%Following the step by step argument of last paragraph of the proof of \cite[Proposition 3.4, p. 2035]{viv2}, we get $|Aut G|>|G|^2$. This is a contradiction to the \cite[Lemma 3.4, p. 655]{rjpf}. This proves the proposition.
\end{proof}

\begin{proof}[Proof of the Main Theorem]
%Let $\mathcal{B}_1=\{S\in \mathcal{T}(G,H)|H_S=\left\{1\right\}\}$ and $\mathcal{B}_2=\{S\in \mathcal{T}(G,H)|H_S=H\}$. Each member of $\mathcal{B}_1$ is subgroup of $G$. Let $S=\{1,x_1,x_2,x_3, \ldots, x_{n-1}\} \in \mathcal{B}_1$. Since $|S| \geq 4$, it contains nontrivial distinct elements $x_1, x_2 $ and $x_3=x_1 x_2$. Let $h \in H \setminus\{1\}$. Let $S'=S \setminus \{x_3\} \cup \{ hx_3\}$. Then $H_{S'} \neq \left\{1\right\} $. Hence $\langle S' \rangle =G $ (Proposition \ref{p1}(iii)) and so $S' \in \mathcal{B}_1$.
%Hence $S \mapsto S'$ defines an injective map from $\mathcal {B}_1$ to $\mathcal {B}_2$.

Let $\mathcal {A}_1, \mathcal {A}_2, \mathcal {A}_3$ and $\mathcal {A}_4$ denote the distinct isomorphism classes in $\mathcal{T}(G,H)$. Then by Proposition \ref{p10} , $m^{n-1}=\left|\mathcal{T}(G,H)\right|<4(\frac{m^{n-1}}{4})$, a contradiction.%Suppose that $\mathcal {A}_1$ is a collection of subgroups of $G$, that is $\mathcal {B}_1=\mathcal {A}_1$ and $\mathcal {B}_2=\mathcal {A}_2 \cup \mathcal {A}_3 \cup \mathcal {A}_4$. Therefore by Proposition \ref{p10} and observation in above paragraph $|\mathcal {B}_2| <|\mathcal {A}_2|+|\mathcal {A}_3|+|\mathcal {A}_4| < \frac{3}{8} |\mathcal{T}(G,H)|$ and $|\mathcal {B}_1| \leq |\mathcal {B}_2| < \frac{3}{8} |\mathcal{T}(G,H)|$. But in this case $|\mathcal{T}(G,H)|=|\mathcal {B}_1|+|\mathcal {B}_2|<\frac{6}{8}|\mathcal{T}(G,H)|$, a conradiction.

%Assume that $\mathcal{B}_1=\mathcal {A}_1 \cup \mathcal {A}_2 \cup \mathcal {A}_3$ and $\mathcal{B}_2=\mathcal {A}_4$. Then in this case $|\mathcal {B}_2|=|\mathcal {A}_4| < \frac{1}{8}|\mathcal{T}(G,H)|$ and $ |\mathcal {B}_1| \leq |\mathcal {B}_2| <\frac{1}{8}|\mathcal{T}(G,H)|$. Which will again give a conradiction. Similarly other cases are not possible. 
\end{proof}


\begin{thebibliography}{99}

\bibitem{bg} J. M. Burns and B. Goldsmith, {\it Maximal Order Abelian Subgroup of Symmetric Group}, Bull. London Math. Soc. 21 (1989), 70-72.

\bibitem{wb} William Burnside, {\it Theory of Groups of Finite Order}, Cambridge Univ. Press, (1897).

\bibitem{cam} Cameron, Peter J., {\it http://www.maths.qmul.ac.uk/pjc/preprints/transgenic.pdf}

\bibitem{gap} The GAP Group, , {\it GAP-Groups, Algorithms, and Programming, Version 4.4.10, 2007}, http://www.gap-system.org.

\bibitem{viv1} V.K. Jain, R.P. Shukla, {\it On the isomorphism classes of transversals}, Comm. Alg. 36 (5) (2008), 1717-1725.
 
\bibitem{viv2} V.K. Jain, R.P. Shukla, {\it On the isomorphism classes of transversals-II}, Comm. Alg. 39 (6) (2011), 2024-2036. 

\bibitem{viv3} V.K. Jain, {\it On the isomorphism classes of transversals-III}, 	arXiv:1112.5530v1 [math.GR].

\bibitem{vk} Vipul Kakkar, R. P. Shukla, {\it Some Characterizations of a Normal Subgroup of a Group}, arXiv:1202.5626v2 [math.GR].


\bibitem{rltr}R. Lal,
{\em Transversals in Groups},
J. Algebra 181 (1996), 70-81.

%\bibitem{rldt} R. Lal, {\it Some Problems on Dedekind-Type Groups},
%J. Algebra
%181 (1996), 223-234.

\bibitem{rjpf} R. Lal and R. P. Shukla,
{\it Perfectly stable  subgroups of finite groups},
Comm. Algebra 24(2) (1996), 643-657.

%\bibitem{rjtm} R. Lal and R. P. Shukla,
%{\it A characterization of Tarski monsters}, to appear in
%Indian J. Pure Appl. Math.

%\bibitem{rjnd} R. Lal and R. P. Shukla,
%{\it Transversals in Non-discrete groups}, Proc. Indian Acad. Sci. (Math. Sci.)
%Vol. 115 no. 4 (2005), 1-7.

%\bibitem{bdm} Brendan D. McKay et. al. {\it Small Latin Squares, Quasigroups and Loops}, J. Combin. Des. 15 (2007), no. 2, 98–119.

%\bibitem{jnb} J. NeubÄuser, {\it Die UntergruppenverbÄande der Gruppen der Ordnung · 100 mit Ausnahme der
%Ordnungen 64 und 96}, Habilitationsschrift an der UniversitÄat Kiel.

\bibitem{rob} D. J. S. Robinson,
{\it A Course in the Theory of Groups},
Springer-Verlag, New York, 1996.

\bibitem{rpsc} R. P. Shukla,
{\it Congruences In Right Quasigroups and General Extensions},
Comm. Algebra 23(7) (1995), 2679-2695.

\bibitem{smth} J. D. H. Smith,
{\it An Introduction to Quasigroups and Their Representations},
Chapman and Hall/CRC, Boca Raton, FL, 2007.

\bibitem{suz} Michio Suzuki, {\it Group Theory I}, Springer-Verlag, New York, 1982.

\bibitem{lw} Louis Weisner,
{\it On the Sylow Subgroups of the Symmetric and Alternating Groups},
American Journal of Mathematics, 47(2)(Apr., 1925), 121-124.
\end{thebibliography}
\end{document}